\documentclass[12pt]{article}
\usepackage{amssymb, amsthm,amsmath, bm, bbm, mathrsfs, graphicx}
\usepackage{latexsym}
\usepackage{multicol}
\renewcommand{\P}{\mathbb{P}}
\newcommand{\E}{\mathbb{E}}
\newcommand{\I}{\mathbb{I}}
\newcommand{\iP}{\rm PC}
\newtheorem{theorem}{Theorem}
\newtheorem{lemma}{Lemma}
\newtheorem{definition}{Definition}
\newtheorem{remark}{Remark}
\newtheorem{corollary}{Corollary}
\newtheorem{example}{Example}
\begin{document}

%
\begin{center}

\textbf{ \Large Some Pitman Closeness Properties Pertinent to Symmetric Populations}

\vspace{0.5cm}

{Mohammad Jafari Jozani$^{a,}$\footnote{Corresponding author: m$_{-}$jafari$_{-}$jozani@umanitoba.ca}, N. Balakrishnan$^b$, Katherine F. Davies$^a$}

\vspace{0.5cm}

{\it $^a$ University of Manitoba, Department of Statistics, Winnipeg, Manitoba, \\Canada R3T 2N2 }

{\it $^b$ McMaster University, Department of Mathematics and Statistics, Hamilton, Ontario, \\ Canada L8S 4K1 }
\end{center}

\begin{abstract}
\noindent In this paper, we focus on Pitman closeness probabilities  when the  estimators are symmetrically distributed  about the unknown parameter $\theta$.  We first consider two symmetric  estimators  $\hat{\theta}_1$ and $\hat{\theta}_2$ and obtain necessary and sufficient conditions for $\hat{\theta}_1$ to be Pitman closer to the common median $\theta$ than $\hat{\theta}_2$. We then establish some properties in the context of estimation under Pitman closeness criterion.  We define a Pitman closeness probability which measures the frequency with which an individual order statistic is Pitman closer  to $\theta$ than some symmetric estimator. We show that, for symmetric populations,  the sample median is Pitman closer to the population median than any other symmetrically distributed estimator of $\theta$. Finally, we discuss the use of Pitman closeness  probabilities in the determination of an optimal ranked set sampling scheme (denoted by RSS) for the estimation of the population median when the underlying distribution is symmetric. We show that the best RSS scheme from symmetric populations in the sense of Pitman closeness is the median and randomized median  RSS for the cases of odd and even sample sizes, respectively.

\end{abstract}

\section{Introduction}

The concept of Pitman's measure of closeness, simply referred to as Pitman Closeness, was introduced  by Pitman (1937). Over the years, it has been a competing criterion in the choice of ``efficient estimators" along with other criteria such as unbiasedness, minimum variance, and minimum mean squared error. As a probability, Pitman closeness measures the frequency with which one estimator is closer to the value of a parameter than another competing estimator within the same class of estimators. More precisely, we have the following definition.
%

\begin{definition}
Let $\hat{\theta}_1$ and $\hat{\theta}_2$ be univariate estimators of a real-valued parameter $\theta$ based on a sample of size $n$. Pitman Closeness (PC) is then defined as
\begin{equation}  \label{eqn:pcprob}
\iP(\hat{\theta}_1,\hat{\theta}_2|\theta, n)=\P_{\theta}\left(|\hat{\theta}_1-\theta|<|\hat{\theta}_2-\theta|\right).
\end{equation}
\end{definition}
Using the PC probability in (\ref{eqn:pcprob}), we can state that the estimator $\hat{\theta}_1$ is Pitman closer to $\theta$  than $\hat{\theta}_2$ if $\iP(\hat{\theta}_1,\hat{\theta}_2|\theta, n) \geq \frac{1}{2}$ for all $\theta$ in the parameter space $\Theta$, with strict inequality holding for at least one $\theta$. For further details on the concept of Pitman closeness and its applications, one may refer to Keating et al.  (1993).



Recently, considerable discussion has taken place on the use of Pitman closeness as a criterion in the context of ordered data as estimators. The basic work in this direction started with Balakrishnan et al. (2009) who established that the sample median is Pitman closest to the population median among all order statistics in a sample.  Subsequently, Pitman closeness has been used in estimating the  population parameters such as  quantiles and median using order statistics, records and censored data. For a list of  most recent works in this direction we refer to Volterman et al.  (2012) and the references cited therein. 

In this paper, we study Pitman closeness for  symmetrically distributed estimators about the same median $\theta$.  We start with two such estimators $\hat{\theta}_1$ and $\hat{\theta}_2$.  Without loss of generality, through out the paper,  we will simply refer to  $\hat{\theta}_1$ and $\hat{\theta}_2$ as symmetrically distributed random variables $X$ and $Y$.   In Section \ref{rss1}, we obtain necessary and sufficient conditions for $X$ to be Pitman closer to $\theta$ than $Y$.  The results are augmented by several examples. In Section \ref{rss2}, we carry out a Pitman closeness comparison between order statistics and symmetric estimators. In this framework, we define a Pitman closeness probability and establish some properties which suggest optimal estimation of the population median in the case of symmetric distributions. In Section \ref{ss1}, we discuss some optimal ranked set sampling schemes for symmetric populations based on Pitman closeness for the estimation of the population median. Finally, in Section \ref{conc},  we  make some concluding remarks.


\section{Pitman Closeness and Symmetric Variables} \label{rss1}

Assume that $X$ and $Y$ are  symmetrically distributed random variables about the same unknown median $\theta\in\Theta$. Suppose that $X\sim F_X(x; \theta)$  and $Y\sim G_Y(y; \theta)$, and that their probability density functions (pdfs) are denoted by $f_X(x; \theta)$ and $g_Y(y; \theta)$, respectively.   Let  $S_X$ and $S_Y$ denote the supports of $X$ and $Y$, respectively.  Without loss of generality, we assume that $\theta=0$, $S_X= (-a, a)$ and $S_Y= (-b, b)$ with $0< a, b\leq \infty$; if not, for example, when $S_X= (r_1, r_2)$, we can define $X'= X-\frac{r_1+r_2}{2}$ with $S_{X'}=(-a, a)$, $a= \frac{r_2-r_1}{2}$ and $\theta'=0$.  When  $\theta=0$, for simplicity, we use $F_X(x)$,  $G_Y(y)$,  $f_X(x)$ and $g_Y(y)$ to denote the corresponding cdfs and pdfs. 

\begin{lemma}\label{lem:pitman-SS}
 Suppose $X$ and $Y$ are independent and symmetrically distributed about $\theta=0$, and that $S_X=(-a, a)$ and $S_Y=(-b, b)$ represent the supports of $X$ and $Y$, respectively, with $0< a , b\leq\infty$. 
 \begin{itemize}
 \item[(i)] When $a<b$, 
 $X$ is Pitman closer to $\theta$ than $Y$ if and only if  
 \begin{eqnarray}\label{ineq-condition}\int_{0}^{a} F_X(t) g_Y(t)\, dt \geq G_Y(a)-\frac{5}{8}; \end{eqnarray}
\item [(ii)]
 When $a\geq b$, $X$ is Pitman closer to $\theta$ than $Y$ if and only if   
  \begin{eqnarray}\label{same-support}\int_{0}^{b} F_X(t) g_Y(t)\, dt \geq \frac{3}{8}. \end{eqnarray}
\end{itemize}
\end{lemma}
\begin{proof}  To show $(i)$, consider 
\begin{eqnarray*}
\P_0(|X| <|Y|)&=& \P_0(X-Y<0, X+Y>0) +\P(X-Y>0 , X+Y<0)\\&=& \int_{0}^{b} \{ F_X(t) - F_X(-t)\} g_Y(t)\, dt+ \int_{-b}^{0}\{F_X(-t) - F_X(t) \} g_Y(t)\,dy\\&=& 2 \int_{0}^{b}\{F_X(t)- F_X(-t) \} \,g_Y(t)\, dt\\&=& 2\int_{0}^a \{F_X(t) - (1-F_X(t))\} g_Y(t)\, dt + 2\int_{a}^{b}  \{ 1-0\} g_Y(t)\, dt   \\&=&2 \int_{0}^{a} \{ 2 F_X(t)-1 \}\,g_Y(t)\, dt + 2 \left\{G_Y(b)- G_Y(a) \right\}\\&=& 4\int_{0}^{a} F_X(t) g_Y(t)\, dt - 4G_Y(a) +3.
\end{eqnarray*}
Thus, $\P_0(|X| <|Y|) \geq \frac12$ if and only if $\int_{0}^{a} F_X(t) g_Y(t)\, dt \geq G_Y(a)-\frac{5}{8}$.  

To show $(ii)$, we have 
\begin{eqnarray*}
\P_0(|X| <|Y|)&=&  2 \int_{0}^{b}\{F_X(t)- F_X(-t) \} \,g_Y(t)\, dt\\&=& 4\int_{0}^{b} F_X(t) g_Y(t)\, dt - 1,
\end{eqnarray*}
and so  $\P_0(|X| <|Y|) \geq \frac12$ if and only if  $\int_{0}^{b} F_X(t) g_Y(t)\, dt \geq \frac{3}{8}$.
\end{proof}

\vspace{0.5cm}

We can also present the following equivalent necessary and sufficient conditions  for $X$ to be Pitman closer to $\theta$ than $Y$, which in some cases are more convenient to work with when  compared to  \eqref{ineq-condition} and \eqref{same-support}  (as in the case of Example \ref{beta-generated} below). 

\begin{corollary}\label{cor-equivalent} 
Assume the conditions of Lemma \ref{lem:pitman-SS} to be true. \begin{itemize} \item [(i)] When $a>b$, 
$X$ is Pitman closer  to $\theta$ than $Y$ if and only if \begin{eqnarray}\int_{0}^{b} G_Y(t) f_X(t)\, dt \leq F_X(b)-\frac{5}{8};\end{eqnarray}

\item[(ii)] When $a\leq b$,  $X$ is Pitman closer to $\theta$ than $Y$ if and only if 
 \begin{eqnarray}\int_{0}^{a} G_Y(t) f_X(t)\, dt \leq \frac{3}{8}.\end{eqnarray} 
\end{itemize}
\end{corollary}
\begin{proof}
The proof is similar to that of Lemma \ref{lem:pitman-SS} and is therefore omitted for brevity.
\end{proof}

\begin{example}
Let $X\sim U(\theta-a, \theta+a)$ and $Y\sim N(\theta, 1)$, $\theta\in \mathbb{R}$, and let $\phi(\cdot)$ and $\Phi(\cdot)$ denote the pdf and cdf of  $Y-\theta$, respectively.   It can be shown that the random variable $X$ is Pitman closer to $\theta$ than $Y$ when $a\leq a_0\simeq1.47$. To this end,  using Lemma \ref{lem:pitman-SS},  we have $X$ to be Pitman closer to $\theta$  than $Y$ if and only if $$ \int_0^a \frac{(t+a)}{2a} \phi(t)\, dt \geq \Phi(a)-\frac{5}{8}, $$
or   $$\frac {1}{2a} \left\{\frac{1}{\sqrt{2\pi}}-\phi(a) \right\} + \frac 12 \left\{ \Phi(a) - \frac 12 \right\}\geq  \Phi(a)-\frac 58, $$
which is equivalent to $h(a) \leq 0$, where $h(a) =a(\Phi(a) -\frac 34) +  \phi(a)-\phi(0)$. It is easy to check that $h(0)=0$, $\lim_{a\to \infty}h(a)= \infty$,  and also  $$h'(a)=\frac{d}{da}h(a)= \Phi(a)- \frac 34\quad\ \mbox{with} \quad h''(a)= \frac{d^2}{da^2}h(a)= \phi(a)>0. $$
Hence, $h(a)$ is a convex function of $a>0$, and so $h(a)\leq 0$ for all $a\leq a_0\simeq1.47$, where $a_0$ is obtained numerically such that $h(a_0)=0$.  Also, for any $a\geq a_0$, $Y$ is Pitman closer to $\theta$ than $X$.

\end{example}


In the following  examples, we introduce classes of random variables and study Pitman closeness among the members of the families when the parent distribution is symmetrically distributed about the population median $\theta$.

\begin{example} \label{beta-generated}   Consider the class $\mathcal{C}_{\alpha}=\{ X_{\alpha}: \alpha\geq 0\}$ of random variables $X_{\alpha}$ having   pdf 
\begin{eqnarray}\label{falpha}
f_{X_{\alpha}}(x; \theta)= \frac{1}{B(\alpha+1, \alpha+1)} f_X(x; \theta) [F_X(x; \theta)]^{\alpha} [1-F_X(x; \theta)]^{\alpha},\end{eqnarray}
where $B(r, s)=\frac{\Gamma(r) \Gamma(s)}{\Gamma(r+s)}$ ($r, s>0$) is the complete beta function, and   $F_X(x; \theta)$  is  the ``parent distribution" of the family. Note that the class $\mathcal{C}_{\alpha}$  is a subclass of the general class of beta-generated distributions introduced by Jones (2004). It can be easily shown that if the parent distribution is symmetric about $\theta$, then  $X_{\alpha}\in\mathcal{C}_{\alpha}$ is also symmetrically distributed about $\theta$. Now, we show that  any $X_{\alpha}\in\mathcal{C}_{\alpha}$ ($\alpha>0$) is Pitman closer to $\theta$ than $X_0 \equiv X$. To this end,  without loss of generality, let us take $\theta=0$. We need to show that $\pi_{\alpha}= \P_0(|X_{\alpha}| <|X|) \geq \frac 12$ for all $\alpha>0$. Since $X$ and $X_{\alpha}$ have the same support,  using Part $(ii)$ of Corollary \ref{cor-equivalent}, we only need to show that    $\int_{0}^{\infty} F_X(x) f_{X_{\alpha}}(x) \, dx \leq \frac{3}{8}.$ For this purpose, let us consider
\begin{eqnarray*}
\int_{0}^{\infty}  F_X(x) f_{X_{\alpha}}(x)\, dx&=& \frac{\Gamma(2\alpha+2)} {\Gamma^2(\alpha+1)}\int_0^{\infty} f_X(x) [F_X(x)]^{\alpha+1} [1-F_X(x)]^{\alpha}\, dx\\&=& \frac{1}{2} \int_{\frac12}^{1} \frac{1}{B(\alpha+2, \alpha+1)} t^{\alpha+1} (1-t)^{\alpha}\, dt\\&:=&H(\alpha).
\end{eqnarray*}
 Now, the result follows from the fact that $H(\alpha)$ is a decreasing function of $\alpha$ (for $\alpha\geq 0$), with $H(0)=\frac{1}{2 B(2, 1)} \int_{\frac12}^1 t\, dt= \frac 38 $. Also,  $$\pi_{\alpha}= 2- 4 H(\alpha)  $$
is evidently an increasing function of $\alpha$. Consider the special case of odd sample size, say $n=2m-1$. In this case, the established result reveals that the sample median $X_{m:n}$ is always Pitman closer to $\theta$ than every other  sample observation and that the Pitman closeness probability of  $X_{m:n}$ being  closer to $\theta$ than $X$ is an increasing function of $m$, a result established earlier by Balakrishnan et al. (2009).
\end{example}

\begin{remark}
A well-known family of distributions that fits in the framework of Example \ref{beta-generated} is the Type-III generalized logistic family of distributions; see Balakrishnan  (1992). We then have the result that the Type-III logistic random variable $X_{\alpha}$ is Pitman closer to the population median $\theta$ than the logistic random variable $X$, and that the Pitman closeness probability increases with $\alpha$.
\end{remark}

\begin{example}\label{extended-falpha}
 One can easily extend the result in Example \ref{beta-generated} to a more general class of random variables  $\mathcal{C}_{\boldsymbol{\alpha}}= \{ X_{\boldsymbol{\alpha}}:  \boldsymbol{\alpha}=(\alpha_1, \ldots, \alpha_k), \alpha_i>0, i\in{1, \ldots, k}\}, $ where $X_{\boldsymbol{\alpha}}$ has the following mixture distribution:
 $$f_{X_{\boldsymbol{\alpha}}}(x; \theta)=  \sum_{i=1}^k  \pi_i\, f_{X_{\alpha_i}}(x; \theta), $$
where $f_{X_{\alpha_i}}(x; \theta)$ is as defined in \eqref{falpha}. Now, as in Example \ref{beta-generated},  it can be shown that each member of the class $\mathcal{C}_{\boldsymbol{\alpha}}$ is Pitman closer to $\theta$ than $X_{\textbf 0}$.  
\end{example}

 
\begin{example}
As in Example \ref{beta-generated},  let us consider the class of random variables $X_{\alpha}$, but with $\alpha \in (-1,  0]$; i.e., $\mathcal{C}^*_{\alpha}=\{X_{\alpha}: \alpha \in (-1,0]\}$. In this case, we can  easily verify that  the random variable $X_0\equiv X$ is Pitman closest to $\theta$ within the class   $\mathcal{C}^*_{\alpha}$, i.e.,  $X$ is Pitman closer to $\theta$ than any $X_{\alpha}\in\mathcal{C}^*_{\alpha}$. For example, consider the case when $\alpha=-\frac 12$. In this case, we have
$$f_{X_{-\frac 12}}(x; \theta) = \frac{f_{X}(x; \theta)}{\pi \sqrt{F_X(x;\theta) (1-F_X(x; \theta))}},   $$ and 
 $$ \P_0(|X|<|X_{-\frac12}|) = \frac{4}{\pi} \int_{\frac 12}^1 \sqrt{\frac{t}{1-t}} dt-1 = 0.6366 \geq \frac 38,$$ 
which means that $X$ is Pitman closer to $\theta$ than $X_{-\frac12}$. 
\end{example}

\begin{remark}
An extension of this result in the form of Example \ref{extended-falpha} could be presented here as well.
\end{remark}

In the following lemma, we present a sufficient condition for the results in Lemma  \ref{lem:pitman-SS} to  hold.  

\begin{lemma}\label{suff-condition} Suppose $X$ and $Y$ are independent and symmetrically distributed about $\theta$ (which can be taken as 0 without loss of generality). Suppose the supports of $X$ and $Y$ are $S_X=(-a, a)$ and $S_Y=(-b, b)$, with $0<a\leq b\leq \infty$, respectively. Then,  a sufficient condition for $X$ to be Pitman closer to $\theta$ than $Y$ is that  $\P_{\theta}(X-\theta \leq t) \geq \P_{\theta}(Y-\theta \leq t)$ for all $t\geq 0$.
\end{lemma}
\begin{proof}
Using the condition that $F_X(t)\geq G_Y(t)$, for all $t\geq 0$, we have
\begin{eqnarray*}
\P_{\theta}(|X-\theta| <|Y-\theta|)&=& \P_0(|X| <|Y|)\\&=&  4\int_{0}^{a} F_X(t) g_Y(t)\, dt -4G_Y(a)+3 \\&\geq& 4\int_{0}^{a} G_Y(t) g_Y(t)\, dt -4G_Y(a)+3\\&=& 4\frac{G^2_Y(t)}{2}\bigg|_0^{a} -4G_Y(a)+3\\&=& 2 (G_Y(a) -1)^2 + \frac 12 \\&\geq&  \frac 12,
\end{eqnarray*}
as required.
\end{proof}

It is of interest to mention that the sufficient condition stated in Lemma \ref{suff-condition} is equivalent to a condition relating to the notion of `peakedness'  of the distributions of $X$ and $Y$ about $\theta$. To this end, let us assume that the condition in Lemma \ref{suff-condition} holds, i.e.,  $\P_{\theta}(\theta-X \leq t) \geq \P_{\theta}(\theta-Y \leq t)$ for all $t\geq 0$. Then, by using the symmetry of  $X$ and $Y$ about $\theta$, we readily have $$\P_{\theta} (X-\theta\leq -t) \leq \P_{\theta}(Y-\theta \leq -t) \quad \text{for all } t\geq 0. $$
Consequently,  we obtain 
\begin{eqnarray*}\P_{\theta}(| X-\theta| \leq t) &=& \P_{\theta}(X-\theta \leq t) - \P_{\theta} (X-\theta\leq -t)\\&\geq& \P_{\theta}(Y-\theta \leq t)-  \P_{\theta}(Y-\theta \leq -t)\\&=& \P_{\theta} (| Y-\theta| \leq t)\end{eqnarray*}
for all $t\geq 0$, meaning that $X$ is more peaked about $\theta$ than $Y$, which leads to the following definition.

\begin{definition}\label{def-peak} Let $X$ and $Y$ be two real-valued random variables. We say that $X$ is more peaked about $\theta$ than  $Y$ if  $\P_{\theta}(|X-\theta|\leq t) \geq \P_{\theta}(|Y-\theta| \leq t)$ for all $t\geq 0$. 
\end{definition}

We then have the  following lemma  which simply states that between two symmetrically distributed random variables about $\theta$, the more peaked random variable is Pitman closer to $\theta$ than the less peaked one. 

\begin{lemma}\label{lem-peaked}
Suppose $X$ and $Y$ are independent and symmetric random variables about $\theta$. Then, if the distribution of $X$ is more peaked about $\theta$ than $Y$ and $S_X\subseteq S_Y$, $X$ is Pitman closer to $\theta$ than $Y$. 
\end{lemma}
%

\begin{remark} Consider the case when $X$ and $Y$ are two symmetrically distributed random variables about the same $\theta$ with finite variances.  One can easily show that a necessary condition for $X$ to be more peaked about $\theta$ than $Y$ is that $Var(X)\leq Var(Y)$. To see this, without loss of generality, let us take  $\theta=0$. Now, since $Var(X)= \E(X^2) = 2\int _0^{\infty} t \{1-F_X(t) + F_X(-t)\}dt$ and $F_X(t)- F_X(-t)\geq G_Y(t)-G_Y(-t)$ for all $t>0$, we immediately have $Var(X)=  2\int _0^{\infty} t \{1-F_X(t)+ F_X(-t)\}dt \leq  2\int _0^{\infty} t \{1-G_Y(t)+ G_Y(-t)\}dt= Var(Y)$.   
\end{remark}

We can present similar results for convex combinations of independent symmetric random variables as follows.

\begin{lemma}\label{lem:linear}
Suppose $X$ and $Y$ are independent random variables each symmetrically distributed about $\theta$ with the same support. Then, if the distribution of $X$ is more peaked about $\theta$ than $Y$,   $\omega X + (1-\omega) Y$ ($0\leq \omega <1$) is Pitman closer to $\theta$ than $Y$. 
\end{lemma}
\begin{proof} To show the result note that 
\begin{eqnarray*}
\P_{\theta}(| \omega X+ (1-\omega)Y-\theta| < |Y-\theta|) &=& \P_{0}(| \omega X+ (1-\omega)Y| < |Y|)\\&=&2 \int_{0}^{\infty} \left[ F_X(t)- F_X( -\frac{2-\omega}{\omega} t)\right]  g_Y(t) dt\\&\geq & 2 \int_0^{\infty} [2 F_X(t) -1] g_Y(t) dt\\ &\geq& 2 \int_0^{\infty} [2 G_Y(t) -1] g_Y(t) dt\\&=&\frac12,
\end{eqnarray*}
as required.
\end{proof}

\begin{corollary}\label{cor:linear} Suppose $X_1$ and $X_2$ are i.i.d.\ random variables each symmetrically distributed about $\theta$. Then,  $\omega X_1 + (1-\omega) X_2$ is Pitman closer to $\theta$ than $X_1$ for all $0\leq \omega <1$. In particular $\bar{X}_{2n}$ is always Pitman closer to $\theta$ than $\bar{X}_n$, $n\geq 1$, where $\bar{X}_j=\sum_{i=1}^j X_i$.\end{corollary}

In Corollary \ref{cor:linear}, one can easily show that $\frac{X_1+X_2}{2}$ is the Pitman closest estimator of $\theta$ within the class of estimators $\omega X_1 + (1-\omega) X_2$, $0\leq \omega<1$. To this end, note that 
\begin{eqnarray*}
&& \P_{\theta}\left(\left| \frac{X_1+ X_2}{2}-\theta \right| \leq | \omega X_1 + (1-\omega) X_2 -\theta|\right) \\&=& \P_{0}\left( \left| \frac{X_1+ X_2}{2}\right| \leq | \omega X_1 + (1-\omega) X_2|\right) \\&=& \begin{cases} 2 \int_{0}^{\infty} \left[  F_X(t) - F_X(\frac{2\omega -3}{2\omega+1} t) \right] f_X(t) dt,&  0\leq \omega<\frac 12,  \\  \, \\ 2 \int_{0}^{\infty} \left[ F_X(t) - F_X(\frac{2\omega +1}{2\omega-3} t)\right] f_X(t) dt,&  \frac 12\leq \omega<1,
\end{cases}\\&\geq &2\int_0^{\infty} [ F_X(t) - F_X(-t)] f_X(t) dt\\&=&\frac 12, 
\end{eqnarray*}
which completes the proof. 


Let us now turn our attention to location-scale families and develop some Pitman closeness results.

\begin{lemma}\label{cor-scale}
Consider  a symmetric location-scale family of continuous  distributions 
$$\mathcal{F}=\left\{  F(x; \theta, \sigma)= G\left(\frac{x-\theta}{\sigma}\right), \text{with $G(\cdot)$ being symmetric about 0, }  x, \theta\in \mathbb{R}, \sigma>0\right\}. $$ Let $X\sim F(\cdot; \theta, \sigma_1)$ and $Y\sim F(\cdot; \theta, \sigma_2)$ with $F\in\mathcal{F}$. Then, $X$ is Pitman closer to $\theta$ than $Y$ whenever $\sigma_1< \sigma_2$. 
\end{lemma}

\begin{proof}
The result follows immediately from Part $(ii)$ of Lemma \ref{lem:pitman-SS} upon using the fact that   $ G(\frac{y}{\sigma_1}) \geq G(\frac{y}{\sigma_2}) $ for all $\sigma_1<\sigma_2$ and $y>0$ when $\theta=0$.
\end{proof}

%

\begin{example}\label{normal-pit}
Let $X_1, \ldots, X_n$ be an i.i.d.\ sample from $N(\theta,  \sigma^2)$ distribution, and $Y_1, \ldots, Y_n$ be another i.i.d.\ sample from    $N(\theta, k \sigma^2)$ with  $k> 0$. Then:
\begin{itemize}
\item[(i)] $\sum_{i=1}^n a_iX_i$ is Pitman closer  to $\theta$ than $\sum_{i=1}^n b_i Y_i$ whenever $ k> \frac{\sum_{i=1}^n a_i^2} { \sum_{i=1}^n b_i^2}$, where $a_i, b_i>0$ with $\sum_{i=1}^na_i= \sum_{i=1}^n b_i=1$;

\item[(ii)] $\bar{X}$ is Pitman closer to $\theta$ than $\bar{Y}$ whenever $k>1$.  
\end{itemize}
\end{example}

In what follows, we establish Pitman closeness results in the setting of Part $(ii)$ of Lemma \ref{lem:pitman-SS} when the two random variables have the same support, i.e.,  $S_X=S_Y$, and by relaxing some of the assumptions on the distributions. 

\begin{corollary}\label{cor:pitman-SS1}
 Suppose $X$ and $Y$ are independent continuous random variables with location parameter $\theta$ and the same support. Moreover, let $Y$ be symmetrically distributed about $\theta$ and that $\P_{\theta}(X<\theta-t) \leq \P_{\theta}(X>\theta+t)$ for all $t >0$. Then, $X$ is Pitman closer to $\theta$ than $Y$ if 
 
 (i)   $\P_{\theta}(X-\theta \leq t)\geq \P_{\theta}(Y-\theta\leq t)$ for all $t\geq 0$, or
 
 (ii) $\P_{\theta}(X-\theta \leq t)\leq \P_{\theta}(Y-\theta\leq t)$ for all $t\leq 0$. 
\end{corollary}
\begin{proof}  To show $(i)$, we have 
\begin{eqnarray*}
\P_{\theta} (|X-\theta| < |Y-\theta|)&=& \P_0(|X| <|Y|)\\&=&  2 \int_{0}^{\infty}\{F_X(y)- F_X(-y) \} g_Y(y) dy\\&\geq&  2 \int_{0}^{\infty} \{  F_X(y)-(1-F_X(y)) \}g_Y(y) dy\\&=& 4\int_{0}^{\infty} F_X(y) g_Y(y) dy -1\\&\geq& 4\int_{0}^{\infty} G_Y(y) g_Y(y) dy -1\\&=& \frac 12,
\end{eqnarray*}
wherein the first  and the second inequalities  follow from the fact that $F_X(-t)\geq 1-F_X(t)$ and  $F_X(t) \geq G_Y(t)$ for $t\geq 0$, respectively. The result in $(ii)$ can be obtained similarly. 
\end{proof}

We now establish some Pitman closeness results about randomized estimators.

\begin{lemma}\label{randomized}
Let $X$, $Y$ and $Z$ be three different estimators of $\theta$. Assume that $X$ and $Y$ are both Pitman closer to $\theta$ than $Z$ and define the randomized estimator $T$ as 
 $$ T=\begin{cases} X, & \text{if~ }W=1, \\ Y, & \text {if~ } W=0, 
 \end{cases}$$
 where $W\sim Bernoulli(\zeta)$, $\zeta\in [0, 1]$, is independent of $X$, $Y$ and $Z$. Then, the randomized estimator $T$ is also Pitman closer to $\theta$ than $Z$.
\end{lemma}
\begin{proof}
We need to show that $\P_{\theta}(|T-\theta| < |Z-\theta|)\geq \frac 12$ for all $\theta \in \Theta$. For this purpose, since $T \overset{d}{=} W X + (1-W) Y$,  we have 
\begin{eqnarray*}
\P_{\theta}(|T-\theta| < |Z-\theta|)&=& \E \left[  \I\left(| W X + (1-W) Y-\theta| < |Z-\theta|\right) \right]\\&=& \E[ W  \I(|X-\theta| < |Z-\theta| ) + (1-W) \I(|Y-\theta| < | Z-\theta|)]\\&=& \zeta \P_{\theta} (|X-\theta| < |Z-\theta| )+ (1-\zeta) \P_{\theta} (|Y-\theta| < | Z-\theta|))\\&\geq& \zeta\left(\frac 12\right) + (1-\zeta)\left(\frac 12\right)= \frac 12, 
\end{eqnarray*}
which completes the proof.
\end{proof}

\begin{remark} \label{rem:rand}
Suppose an estimator $\delta_0$ is Pitman closer to $\theta$ than a competing estimator $\delta$ with probability $\pi_0=\P_{\theta}(| \delta_0-\theta|  < |\delta-\theta|)= \frac 12 + a$, $a\in (0, \frac 12]$. Further, suppose $\delta_1$ is another estimator such that $\delta$ is Pitman closer to $\theta$ than $\delta_1$,  with Pitman closeness probability as $\pi_1=\P_{\theta}(| \delta_1-\theta|  <  |\delta-\theta|)=\frac{1}{2}-b$, $b \in \left(0,\frac{1}{2}\right]$. Now, an interesting question that arises is whether it would be possible to form a randomized estimator using $\delta_0$ and $\delta_1$ such that the new estimator is  Pitman closer to $\theta$ than $\delta$. To answer this question, we consider the randomized estimator $\delta^*\overset{d}{=} W \delta_1 + (1-W) \delta_0$, where $W\sim Bernoulli(\zeta)$  independently of $\delta_0$ and $\delta_1$, and find that   
$$\P_{\theta} (| \delta^*-\theta| <| \delta-\theta|)= \zeta \left(\frac{1}{2}-b\right) + (1-\zeta) \left(a+ \frac 12\right).  $$
Hence, any randomized estimator $\delta^*$, with randomization probability $\zeta \in [0, \frac{a}{a+b}]$, will be Pitman closer to $\theta$ than $\delta$. 
\end{remark}

\section{Pitman Closeness and Symmetric Estimators} \label{rss2}

In this section, we establish some Pitman closeness properties of symmetric estimators. Let $X\sim F_X(\cdot; \theta)$ and  $Y\sim G_Y(\cdot; \theta)$ be two independent and absolutely continuous random variables which are symmetrically distributed about $\theta$ with pdfs $f_X(\cdot; \theta)$ and $g_Y(\cdot; \theta)$, respectively. Let $X$ and $Y$ have the same support,  $X_{1:n}, \ldots, X_{n:n}$ be the order statistics from a random sample of size $n$ from $F_X(\cdot; \theta)$, and $F_{i:n}(\cdot; \theta)$ denote the cdf of $X_{i:n}$, $i=1, \ldots, n$.  Then, it is well-known that (see Arnold et al. (1992) and David and Nagaraja (2003)) 
\begin{eqnarray*} 
F_{i:n}(x;\theta)=\sum_{r=i}^n {{n} \choose {r}} \left[F_X(x;\theta) \right]^{r} \left[1-F_X(x;\theta) \right]^{n-r}. 
\end{eqnarray*}

Now, let 
\begin{eqnarray} \label{pin1}
\pi_{i:n}= \P_{\theta}(|X_{i:n}-\theta| < | Y-\theta|), \quad i=1, \ldots, n, \nonumber
\end{eqnarray}
be the Pitman closeness probability that $X_{i:n}$ is closer to $\theta$ than $Y$. 

\begin{theorem} \label{pin2}

 For $i=1, \ldots, n,$ we have
\begin{eqnarray}\label{pi-final}
  \pi_{i:n}= 2 \int_{\frac 12}^1 g^*(u) \left\{\sum_{r=i}^n \binom{n}{r} u^r (1-u)^{n-r} - \sum_{r=i}^n \binom{n}{r} u^{n-r} (1-u)^{r}  \right\}\, du, \nonumber
\end{eqnarray}
where $g^*(u)= \frac{g_Y(F^{-1}_X(u))}{f_X (F^{-1}_X(u))}$ for $u \in (\frac{1}{2},1]$.

\end{theorem}

\begin{proof}
To show the result, taking $\theta$=0 without loss of generality, note that
\begin{eqnarray}\label{pi-i0}
\pi_{i:n}&=& \P_{\theta} (| X_{i:n} -\theta| < |Y-\theta|)\nonumber\\&=& \P_0(|X_{i:n}| < | Y|)\nonumber \\&=& \P_0(Y< X_{i:n}<-Y, Y<0) + \P_0(-Y< X_{i:n}<Y, Y>0).
\end{eqnarray}
We now obtain an analytical expression for $\pi_{i:n}$ by simplifying the two probabilities on the RHS of (\ref{pi-i0}). Firstly, we find
\begin{eqnarray}\label{secondpart}
&& \P_0(-Y< X_{i:n}<Y, Y>0) \nonumber\\ &=& \int_{0}^{\infty} g_Y(t) \{F_{i:n}(t) - F_{i:n}(-t) \}\, dt\nonumber\\&=& \int_{0}^{\infty} g_Y(t) \left\{\sum_{r=i}^n \binom{n}{r} [F_X(t)]^r [1-F_X(t)]^{n-r} - \sum_{r=i}^n \binom{n}{r} [F_X(-t)]^r [1-F_X(-t)]^{n-r} \right\}\, dt \nonumber\\ 
&=& \int_{0}^{\infty} g_Y(t) \left\{\sum_{r=i}^n \binom{n}{r} [F_X(t)]^r [1-F_X(t)]^{n-r} - \sum_{r=i}^n \binom{n}{r} [F_X(t)]^{n-r} [1-F_X(t)]^{r} \right\}\, dt\nonumber \\
&=& \int_{\frac 12}^1 g^*(u) \left\{\sum_{r=i}^n \binom{n}{r} u^r (1-u)^{n-r} - \sum_{r=i}^n \binom{n}{r} u^{n-r} (1-u)^{r}  \right\}\, du, 
\end{eqnarray}
where  the last equality is obtained by setting $u=F_X(t)$, and 
$$g^*(u)= \frac{g_Y(F_X^{-1}(u))}{f_X (F_X^{-1}(u))}.  $$
We similarly find
\begin{eqnarray}\label{firstpart}
 \P_0(Y< X_{i:n}<-Y, Y<0) &=& \int_{-\infty}^{0} g_Y(t) \{F_{i:n}(-t) - F_{i:n}(t) \}\, dt\nonumber\\&=& \int_{0}^{\infty} g_Y(t) \{F_{i:n}(t) - F_{i:n}(-t) \}\, dt\nonumber\\&=&\P_0(-Y<X_{i:n}<Y, Y>0).\end{eqnarray}
 Using \eqref{secondpart} and \eqref{firstpart}, we obtain the required result.
 \end{proof}

 Let $f_{\alpha, \beta}(t)$ denote the pdf of a Beta($\alpha,\beta$) random variable with density
 \begin{eqnarray*}
 f_{\alpha, \beta}(t)=\frac{1}{B(\alpha,\beta)} t^{\alpha-1} (1-t)^{\beta-1},  \quad 0<t<1.
 \end{eqnarray*}
 Then, it is well-known that 
\begin{equation*}
\sum_{r=i}^n {n \choose r} u^r(1-u)^{n-r}=\int_0^u \frac{1}{B(i,n-i+1)}t^{i-1}(1-t)^{n-i}dt=\int_0^u f_{i,n-i+1}(t)dt.
\end{equation*}
 It is then evident that we can write the expression of $\pi_{i:n}$ in Theorem \ref{pin2} as
\begin{equation} \label{pin3}
\pi_{i:n}=2 \int_{\frac 12}^1 g^*(u) \int_{1-u}^u f_{i,n-i+1}(t)dtdu.
\end{equation}

We now present a symmetry property for the Pitman closeness probabilities $\pi_{i:n}$ which is similar in spirit to the symmetry property established in Balakrishnan et al.\ (2009).

\begin{lemma} \label{pi-symmetric}
The Pitman closeness probabilities  $\pi_{i:n}$ possess a symmetry property, viz., that for $i=1, \ldots, n,$
$$\pi_{i:n}= \pi_{n-i+1:n}.$$  
\end{lemma} 
\begin{proof} The result follows immediately from (\ref{pin3})  upon using the fact that
$$\int_{1-u}^u f_{i,n-i+1}(t)dt=\int_{1-u}^uf_{n-i+1,i}(t)dt.$$
\end{proof}

\begin{lemma}\label{pi-monotone}
Consider the Pitman closeness probabilities $\pi_{i:n}$ defined in Theorem \ref{pin2}. Then, we have
\begin{itemize}
\item[(i)] for $n=2m-1$, $\pi_{m:n}\geq \pi_{i:n}$ for all $i\neq m$; moreover, $\pi_{i:n}$ is increasing in $i$ for $i\in\{1, \ldots, m\}$ and decreasing in $i$ for $i\in\{m, \ldots, n \}$;

\item[(ii)] for $n=2m$, $\pi_{m:n}=\pi_{m+1:n}> \pi_{i:n}$  for all $i\notin \{m, m+1\}$; moreover, $\pi_{i:n}$ is increasing in $i$ for $i=1, \ldots, m$, and decreasing in $i$ for $i=m+1, \ldots, n$.
\end{itemize}

\end{lemma}
\begin{proof}
From Theorem \ref{pin2}, we have 
\begin{eqnarray*}
\pi_{i+1:n}&=& 2 \int_{\frac 12}^1 g^*(u) \left\{\sum_{r=i+1}^n \binom{n}{r} u^r (1-u)^{n-r} - \sum_{r=i+1}^n \binom{n}{r} u^{n-r} (1-u)^{r} \right \}\, du\\&=& 2 \int_{\frac 12}^1 g^*(u) \left\{\sum_{r=i}^n \binom{n}{r} u^r (1-u)^{n-r} - \sum_{r=i}^n \binom{n}{r} u^{n-r} (1-u)^{r} \right. \, du \\ && -\left.   \binom{n}{i} u^i (1-u)^{n-i} + \binom{n}{i} u^{n-i} (1-u)^{i} \right\} du\\&=&
\pi_{i:n} + A_{i:n},
\end{eqnarray*}
where $$A_{i:n}= 2\binom{n}{i}\int_{\frac12}^{1} g^*(u) \{ u^{n-i} (1-u)^{i}- u^i (1-u)^{n-i} \}  du.$$
Now, let us consider the function
$$T_{i, n}(u)=u^{n-i} (1-u)^{i}- u^i (1-u)^{n-i} \quad\text{for $u\in \left[\frac 12, 1\right)$}. $$
It is easy to show that $T_{i, n}(u) >0$ for all $u\in[\frac{1}{2}, 1)$ and $i\leq [\frac{n-1}{2}]$, where $[\cdot]$ is the integer part. Now, since the function $T_{i, n}(u) $ is positive in the interval $[\frac12, 1)$ and the function $g^*(u)$ is obviously positive (being the ratio of two densities), we have that $A_{i:n}$ is positive for those values of $i$. Thus,  for $i=1, \ldots, [\frac{n-1}{2}]$, we have $\pi_{i+1:n} > \pi_{i:n}$. Similarly, for $i\geq [\frac{n-1}{2}]+1$, we see that $T_{i, n}(u)<0$ and so $\pi_{i+1:n} < \pi_{i:n}$. Now,   (a) and (b) follow by setting $n=2m-1$ and $n=2m$, respectively. 
\end{proof}
  
\subsection{Pitman Closeness of a Sample Median from Odd Sample Size}  
  
 Let $X_{m:2m-1}$ be the median in a sample of size $2m-1$ from a distribution symmetric about $\theta$ and let $Y$ be another symmetrically distributed random variable about $\theta$,  independently of $X_{m:2m-1}$.  In the following lemma, we show that the Pitman closeness probability $\pi_{m:2m-1}$ is an increasing function of $m$.
  
 \begin{lemma} \label{monopi}
 The PC probability ${\pi}_{m:2m-1}=\P_{\theta}(|X_{m:2m-1}-\theta|<|Y-\theta|)=\P_0(|X_{m:2m-1}|<|Y|)$ is an increasing function of $m$.
 \end{lemma} 
  
 \begin{proof}
To show the required result, we need to compare ${\pi}_{m:2m-1}$ and ${\pi}_{m+1:2m+1}$ for $m\geq 1$.  To this end, by using the expression in \eqref{pin3}, we have
\begin{eqnarray*}
{\pi}_{m:2m-1}&=& 2 \int_{\frac 12}^1 g^*(u) \int_{1-u}^u \frac{t^{m-1}(1-t)^{m-1}}{B(m,m)}  dt du, 
\end{eqnarray*}
and $\pi_{m+1:2m+1}>\pi_{m:2m-1}$ if  
$$\int_{1-u}^u \frac{t^{m}(1-t)^{m}}{B(m+1, m+1)}  dt > \int_{1-u}^u \frac{t^{m-1}(1-t)^{m-1}}{B(m, m)}  dt,$$
or equivalently
\begin{equation} \label{maineq}\frac{\int_{1-u}^u t^{m}(1-t)^{m}dt}{\int_{1-u}^u t^{m-1}(1-t)^{m-1} dt} >  \frac{B(m+1, m+1)} {B(m, m)}=\frac{\int_{0}^1 t^{m}(1-t)^{m}dt}{\int_{0}^1 t^{m-1}(1-t)^{m-1} dt}.\end{equation}
For this purpose, let us introduce the function
$$K(u)=\frac{\int_{1-u}^u t^{m}(1-t)^{m}dt}{\int_{1-u}^u t^{m-1}(1-t)^{m-1} dt}, \quad \frac 12 \leq u \leq 1  \text{ and } m \geq 1.$$
We observe that \eqref{maineq} is equivalent to  $K(u)>K(1)$. Now, to show  \eqref{maineq},  it suffices to show that $K(u)$ is decreasing in $u \in \left[\frac{1}{2},1\right]$.  By taking the derivative of $K(u)$ with respect to $u$, we find
\begin{eqnarray*}
K'(u)& \propto & 2u^{m}(1-u)^{m} \int_{1-u}^u t^{m-1} (1-t)^{m-1} dt-2u^{m-1}(1-u)^{m-1} \int_{1-u}^u t^{m}(1-t)^{m}dt,
\end{eqnarray*}
and $K'(u)<0$ if
\begin{eqnarray*}
\frac{ \int_{1-u}^u t^{m} (1-t)^{m}dt}{ \int_{1-u}^u t^{m-1}(1-t)^{m-1}dt}>u(1-u).
\end{eqnarray*}
Observe that $\int_{1-u}^u t^{m} (1-t)^{m}dt=\int_{1-u}^{1/2} t^{m-1} (1-t)^{m-1} t(1-t)dt+\int_{1/2}^u t^{m-1} (1-t)^{m-1} t(1-t)dt$ and $t(1-t)>u(1-u)$ since $t(1-t)$ is increasing in $(1-u, \frac 12)$ and decreasing in $(\frac 12, u)$, and so
\begin{eqnarray}
\int_{1-u}^u t^{m} (1-t)^{m}dt&>&u(1-u) \int_{1-u}^{1/2} t^{m-1} (1-t)^{m-1}dt+u(1-u)\int_{1/2}^u t^{m-1} (1-t)^{m-1}dt \nonumber \\
&=&u(1-u) \int_{1-u}^u t^{m-1}(1-t)^{m-1} dt,
\end{eqnarray}
which completes the proof.
\end{proof}

  \begin{remark}
  Of special interest is the case of two competing sample medians. In Theorem \ref{pin2}, upon considering $X_{m':2m'-1}$ as the median of  an independent sample of size $2m'-1$ and taking $\theta$=0 without loss of generality, the Pitman closeness probability of interest in this case is
  $\P_{\theta}(|X_{m:2m-1}-\theta|<|X_{m':2m'-1}-\theta|).$
Using the fact that this probability  is increasing in $m$ for fixed $m' \geq 1$,  we have
  $$\P_{\theta}(|X_{m:2m-1}-\theta|<|X_{m':2m'-1}-\theta|)\left\{\begin{array}{rcl}
  <&\frac 12 & \mbox{\ for\ }  m<m'\\
  =& \frac 12 &\mbox{\ for\ }  m=m'\\
  >& \frac 12&\mbox{\ for\ } m>m',\\
  \end{array}
  \right.
 $$
as should be expected.
 \end{remark}

\section{Efficient RSS for Symmetric Populations  Based on Pitman  Closeness} \label{ss1}

The results in Lemmas \ref{pi-symmetric} and \ref{pi-monotone} lead us to efficient sampling designs for estimating the median in symmetric populations about the same $\theta$ based on RSS. For a review on the concept of RSS, we refer the reader to Chen et al. (2004). We treat the cases of odd and even sample sizes separately and for each case,  we propose an efficient RSS based on Pitman closeness for the estimation of the population median.

\subsection*{(i) $n$ odd, say $n=2m-1$, $m\in\mathbb{N}$}

 In this case, using the established results, a suitable sampling scheme would be the {\em median ranked set sampling}. Under this sampling method, we take $n=2m-1$ independent  simple random samples each of size $n$ from the population and in each sample we measure only the sample median. This sampling scheme results in a median ranked set sample of the form
 $$\{ X_{(m:2m-1)i}, i=1, \ldots, 2m-1\}. $$ 
Now, since the distribution of the sample median $X_{(m:2m-1)i}$ is symmetric about $\theta$, upon using Lemmas \ref{pi-symmetric} and \ref{pi-monotone}, we can conclude that within the class of all symmetrically distributed (about $\theta$) estimators of $\theta$, the median of the median ranked set sample,
$$\delta_{M}(X)= Median \{ X_{(m:2m-1)i}, i=1, \ldots, 2m-1\} $$
 would be the Pitman closest estimator of $\theta$.

\subsection*{ (ii) $n$ even, say $n=2m$, $m\in\mathbb{N}$}
 
 In this case, we propose the use of a {\em randomized median ranked set sampling}. Under this sampling method, we take $n=2m$ independent simple random samples each of size $n$ from the population and in each sample we measure the randomized median defined by
 $$ X^*_i=\begin{cases} X_{(m:2m)i},& \text{if }W_i=1, \\ X_{(m+1:2m)i},& \text {if } W_i=0, 
 \end{cases}$$ 
 where $W_i$ is an independent $Bernoulli(\frac{1}{2})$ variable. Thus, in the $i^{th}$ sample, we nominate $X_{m:2m}$ with probability $\frac12$ and  $X_{m+1:2m}$ with probability $\frac12$. This results in a randomized median ranked set sample of the form 
 $$\{ X^*_i, i=1, \ldots, 2m\},$$
 where $X^*_i \overset{d}{=} W_i X_{(m:2m)i} + (1-W_i) X_{(m+1:2m)i}$.   Here again, it is easy to show that the distribution of $X^*_i$
 is  symmetric about $\theta$ when the parent  distribution $F_X(x; \theta)$ is symmetric about $\theta$.  Now, upon using Lemmas \ref{pi-symmetric} and \ref{pi-monotone}, the randomized estimator
 $$\delta^*_{M}(X) =\begin{cases} X^*_{m:2m}& \text{if }W=1 \\ X^*_{m+1:2m} & \text {if } W=0
 \end{cases},$$
 where $X^*_{i:2m}$ are the order statistics among  $X^*_1, X^*_2,\ldots,X^*_{2m}$ and $W$ is another independent $Bernoulli(\frac{1}{2})$ variable,   would be the Pitman closest estimator of $\theta$ within the class of all symmetrically distributed estimators of $\theta$.

Note that in the above situation, the randomized estimators $X^*_i$, $i=1, \ldots, 2m$,  are all Pitman closer to $\theta$ than $X$. To see this, one can use the result in Example \ref{beta-generated} and argue that the density $f_{X^*_i}(x; \theta)$ can be expressed as 
\begin{eqnarray*}
f_{X^*_i}(x; \theta)&=& \frac{(2m)!}{2 m! (m-1)!} [F_X(x; \theta)]^{m-1} [1-F_X(x; \theta)]^{m-1} \{ 1- F_X(x; \theta) + F_X(x; \theta)\}\\&=& \frac{\Gamma(2m)}{\Gamma(m)\, \Gamma(m)} [F_X(x; \theta)]^{m-1} [1-F_X(x; \theta)]^{m-1} 
\\&=& f_{X_{\alpha}}(x; \theta)\quad\text{with } \alpha=m-1\geq 0.
\end{eqnarray*}
It is also worth mentioning that $X^*_i$ will be Pitman closer to $\theta$ than $X$ even if we do not restrict to the case when $W_i\sim Bernoulli(\frac 12)$.  
Using Remark \ref{rem:rand}, we can extend this  observation to a more general case when $W_i\sim Bernoulli(\zeta)$, $\zeta\in[0,1]$, although in this case the distribution of $X^*_i$ is not symmetric about $\theta$ unless $\zeta=\frac 12. $

\section{Concluding Remarks} \label{conc}
  
Recently, the concept of Pitman closeness   has been discussed extensively in the context of ordered data. In this paper, we have established various Pitman closeness results in the case when the underlying estimators are symmetrically distributed. First, we have studied the Pitman closeness of two independent symmetrically distributed estimators about the same median $\theta$ in the cases when the supports are the same and when they are different. After proving some results,  we have demonstrated their usefulness with a number of examples. Next, we have established a specific result for the  Pitman closeness probability between an individual order statistic and an independent symmetric estimator to the population median. Finally, we have discussed the use of Pitman closeness  probabilities in the determination of an optimal ranked set sampling  scheme for the estimation of the population median. In this case, we have specifically shown that the best scheme in the sense of Pitman closeness is the median ranked set sampling or the randomized median ranked set sampling depending on whether the sample size is odd or even, respectively.

  \section*{Acknowledgements}
The authors thank the Natural Sciences and  Engineering Research Council of Canada for funding their research. The authors would also like to thank Drs.\ Alexandre Leblanc and Brad Johnson for providing comments on an earlier version of this manuscript.

  \end{document}